\documentclass[a4paper,10pt]{article}
\usepackage{amsmath,amsthm,amsfonts,amssymb}
\usepackage{color}
\usepackage{diagrams}

\title{Unipotent Schottky bundles on Riemann surfaces and complex tori}
\date{August, 2012}
\author{Carlos Florentino, Thomas Ludsteck}
\newcommand{\C}{\mathbb{C}}

\newcommand{\N}{\mathbb{N}}
\newcommand{\Z}{\mathbb{Z}}

\theoremstyle{plain}
\newtheorem{satz}{Satz}[section]
\newtheorem{lem}[satz]{Lemma}
\newtheorem{prop}[satz]{Proposition}
\newtheorem{cor}[satz]{Corollary}
\newtheorem{thm}[satz]{Theorem}
\theoremstyle{definition}
\newtheorem{Def}[satz]{Definition}
\newtheorem{thmdefstyle}[satz]{Theorem}

\newtheorem{con}[satz]{Construction}
\newtheorem{rem}[satz]{Remark}

\newarrow{Corresponds}<--->

\def\Ad{\mathsf{Ad}}

\def\hom{\mathsf{Hom}}

\begin{document}

\maketitle

\begin{abstract}
We study a natural map from representations of a free (resp. free
abelian) group of rank $g$ in $GL_{r}(\C)$, to holomorphic vector bundles of
degree zero over a compact Riemann surface $X$ of genus $g$ (resp. complex
torus $X$ of dimension $g$). This map defines what is called a Schottky functor.
Our main result is that this functor induces an equivalence
between the category of unipotent representations of Schottky groups and
the category of unipotent vector bundles on $X$. We also show that, over a
complex torus, any vector or principal bundle with a flat holomorphic connection
is Schottky.  
\end{abstract}

\section{Introduction}\label{sec_introduction}
Let $X$ be a compact K\"ahler manifold. In the context of the so-called
non-abelian Hodge theory, Simpson has established in \cite[Lemma 3.5]{sim} an 
equivalence of categories 
between the category of flat bundles over $X$ and the category of Higgs bundles 
on $X$ which are extensions of stable bundles of degree zero with vanishing 
first and second Chern classes.
In this article, we are particularly interested in unipotent objects, and
establish
equivalences of categories between certain categories of flat unipotent modules 
and 
categories of unipotent holomorphic bundles.

To describe these results, let $X$ be a complex manifold, ${\mathcal O}_{X}$ its
structure sheaf and denote by 
$\pi_{1}(X)$ the fundamental group 
of $X$ (for some base point). In this situation there is a well known functor
$$
\C \pi_{1}(X)\textrm{-mod} \longrightarrow {\mathcal O}_{X}\textrm{-mod},
$$
that assigns to a module over the group ring $\C \pi_{1}(X)$ an ${\mathcal
O}_{X}$-module on $X$. If $M$ is a $\C \pi_{1}(X)$-module and if $M$ is of
finite rank $r$ as a $\C$-module, then $M$ may be identified with a
representation $\rho:\pi_{1}(X) \longrightarrow Aut_{\C}(M)$, 
and after a choice of basis, $Aut_{\C}(M)$ may be identified with $GL_{r}(\C)$.
In this case the assigned ${\mathcal O}_{X}$-module $E_{\rho}$ is a holomorphic
vector bundle
on $X$.    

An interesting subclass of these $\C \pi_{1}(X)$-modules are the so-called
\emph{Schottky modules}. Let $\Sigma$ be a free (or free abelian) group 
and let $\alpha:\pi_{1}(X) \longrightarrow \Sigma$
be a surjective homomoprhism.
A $\C \pi _{1}(X)$-module $M$ is called $\Sigma$-\emph{Schottky} if it is
induced by a $\C \Sigma$-module via the canonical morphism $\tilde{\alpha} : \C
\pi_{1}(X) \longrightarrow \C \Sigma$. 
If one restricts the functor defined above to $\pi_{1}(X)$-modules that
are Schottky, then one obtains a new functor $S$ 
(called the \emph{Schottky functor})
$$
S: \C \Sigma \textrm{-mod} \longrightarrow {\mathcal O}_{X}\textrm{-mod}.
$$
We will also call a ${\mathcal O}_{X}$-module
$E$ on $X$ Schottky, if there is a $\C \Sigma$-module $M$ such that $S(M)$ is
isomorphic to $E$. See Section \ref{sec_Schottky_functors} for more details.

Let us now define ${\bf Un}_{\C\Sigma}$ as the category of unipotent $\C
\Sigma$-modules and 
${\bf Un}_{{\mathcal O}_{X}}$ as the category of unipotent ${\mathcal
O}_{X}$-modules on $X$. The Schottky functor is exact, so induces 
a well defined functor $S:  {\bf Un}_{\C\Sigma} \longrightarrow   {\bf
Un}_{{\mathcal O}_{X}}$ that we will still denote by $S$. Note that a unipotent
${\mathcal O}_{X}$-module of level $r$ is always a vector bundle of rank $r$
(cf. Section \ref{section_unipotent_bd_group}).

Two special cases of Schottky representations are of particular interest in our
paper: 

If $X$ is a \emph{complex torus} of dimension $g$, defined as a quotient $V /
\Lambda$ of a $g$-dimensional $\C$-vector space $V$ by a lattice $\Lambda$, 
then by \cite[S 8 p. 143]{gun} there exists a $\C$-linear basis $e_{1}, \ldots,
e_{g}$ of $V$, and a basis $\lambda_{1},\ldots ,\lambda_{2g}$ of $\Lambda$, such
that $\Pi$, the period matrix of $X$, defined by $\lambda_{i}= \sum _{j=1}^{g}
\pi_{ij}\cdot e_{j}$, is of the form
$$
\Pi = (Z ,I),
$$ 
where $Z \in M(g\times g,\C)$ is a symmetric invertible matrix, 
and $I$ is the identity matrix in dimension $g$.
In this case we can let
$\Sigma$ be the free abelian group with $g$ generators $B_1,...,B_g$
and define $\alpha:\pi_{1}(X)\cong \Lambda\longrightarrow\Sigma$ 
as the surjective homomorphism sending 
$\lambda_i$ to $B_i$, and $\lambda_{g+i}$ to the identity in $\Sigma$,
for all $i=1,...,g$.

Our main result for complex tori is the following (see Section
\ref{section:CaseOfComplexTori}):

\begin{thm}\label{thm_main_for_tori}
Let $X$ be a complex torus of dimension $g$, and let
$\Sigma$ be a free abelian group of rank $g$.
The Schottky functor $S$ induces an equivalence of categories
$$
S:{\bf Un}_{\C \Sigma} \simeq {\bf Un}_{{\mathcal O}_{X}},
$$
between the category of unipotent $\C\Sigma$-modules, and the category of
unipotent ${\mathcal O}_{X}$-modules on $X$.
\end{thm}

One can deduce from Simpson's correspondence that unipotent vector bundles admit
a flat connection. Then we parametrize explicitly the set of representations
that
give rise to isomorphic unipotent vector bundles. This method allows us to show
that the
Schottky functor is essentially surjective.

If $X$ is a compact \emph{Riemann surface} of genus $g$, then let us fix a
canonical basis \cite[III 1.]{fakr} of $\pi _{1}(X)$: elements $a_{1}, \ldots,
a_{g}, b_{1}, \ldots , b_{g}$ 
that generate $\pi _{1}(X)$, subject to the single relation $\prod _{i=1}^{g}
a_{i} b_{i} a_{i}^{-1} b_{i}^{-1} =1$. Let $\Sigma:= F_{g}$
be the free group of $g$ generators $B_{1}, \ldots, B_{g}$, and let $\alpha:
\pi_{1}(X) \longrightarrow \Sigma$ be the homomorphism defined by 
$\alpha(a_{i})=1$, $\alpha(b_{i})=B_{i}$ for $i=1, \ldots, g$.
This is the classical case that justifies the use of the term 
``Schottky functor''.

Our main result for Riemann surfaces, proved in Section
\ref{section_RiemannSurface}, is as follows:

\begin{thm}\label{thm_main_for_riemann_surfaces}
Let $X$ be a compact Riemann surface of genus $g$, and let
$\Sigma$ be a free group of rank $g$.
The Schottky functor $S$ induces an equivalence of categories
$$
S:{\bf Un}_{\C\Sigma} \simeq {\bf Un}_{{\mathcal O}_{X}},
$$
between the category of unipotent $\C\Sigma$-modules, and the category of
unipotent ${\mathcal O}_{X}$-modules on $X$.
\end{thm}

Of course, both theorems above include the case of elliptic curves ($g=1$). 
In the level (or rank) two case, the result basically follows from the fact that
the Schottky functor $S$ induces an isomorphism of \emph{Yoneda Ext} groups 
$S_{*} :YExt^{1}_{\C \Sigma}(\C,\C) \longrightarrow YExt^{1}_{{\mathcal
O}_{X}}({\mathcal O}_{X},{\mathcal O}_{X})$. The general case then results from
the 
induction on the level, as well as  
the fact that $S$ is compatible with Yoneda Ext of $\C\Sigma$-modules and
${\mathcal O}_{X}$-modules. The relevant definitions and techniques from
homological algebra that are needed will be described in Section
\ref{sec_schottky_functors_and_cohomology}.    

We note that, in the broader context of Simpson's correspondence,
the category of Higgs bundles which are extensions of stable bundles of degree
zero with vanishing 
first and second Chern classes, contains all unipotent Higgs bundles
(successive extensions of trivial Higgs line bundles). The latter category
contains the
unipotent vector bundles as a full subcategory by setting the Higgs fields equal
to
zero. In a recent paper S. Lekaus \cite{lek05} has classified those unipotent
$\pi_{1}(X)$-representations that give rise to unipotent Higgs bundles with zero
Higgs fields under the correspondence between flat bundles and their
monodromy representations. The method used for this classification differs from
ours, and it
would be interesting to know whether one can find a relationship between her
result and our
results of Theorem \ref{thm_main_for_tori} and Theorem
\ref{thm_main_for_riemann_surfaces}.

The structure of flat vector bundles on complex tori is well understood, mainly
due to the work of Matsushima \cite{mat} and Morimoto \cite{mor}.
In particular, such a vector bundle admits a flat connection if and only if
every indecomposable component is a tensor product of a line bundle
of degree zero with a flat unipotent vector bundle.
For elliptic curves this result was already known due to the work of Atiyah
\cite{ati}. An analogous result in the case of principal bundles over complex
tori was shown by Biswas-Gomez in \cite{bigo}, and similar descriptions were 
obtained recently by M. Brion \cite{bri} for algebraic homogeneous vector and
principal bundles 
over abelian varieties.

Based on the classification of Matsushima and Morimoto, and Theorem
\ref{thm_main_for_tori} we could then show the following theorem, which
generalizes \cite[Theorem 6]{flo} (see Section
\ref{sec_Flat_and_Schottky_over_tori}).

\begin{thm}\label{thm:Flat-Schottky-vector}Let $E$ be a vector bundle over
a complex torus. Then, $E$ admits a flat holomorphic connection if and only if
$E$ is a Schottky vector bundle.\end{thm}

We also prove the corresponding result in the case of principal bundles. 
Let $G$ be any connected linear algebraic group defined over $\C$. As an
application of Theorem \ref{thm:Flat-Schottky-vector} we have (see Section
\ref{subsec_Principal_Schottky_Bundles}):

\begin{thm}
\label{thm:Flat-Schottky-principal} Let $P$ be a principal $G$-bundle
over a complex torus. Then, $P$ admits a flat holomorphic connection if and only
if $P$ is Schottky.
\end{thm}

The article finishes by putting together Theorems \ref{thm:Flat-Schottky-vector}
and
\ref{thm:Flat-Schottky-principal}, and the main result of \cite{bigo} to show
that not only 
the three notions of flat, Schottky and homogeneous are equivalent for a given
principal bundle 
$P$ over a complex torus, but also, when the first and second Chern class of $P$
both vanish, 
these are equivalent to the corresponding notions for the adjoint vector bundle
$\Ad P$.

Finally, we expect that the methods used in our paper are not only limited to
the
complex case, but may also be applied in the $p$-adic case,
 i.e., to the functors studied by Deninger-Werner \cite{dewe1}, \cite{dewe2} and
Faltings \cite{fal}, or also for the characteristic $p$ case.

\section{Schottky functors}\label{sec_Schottky_functors}

Let $X$ be a connected and locally simply connected topological space. Let $R$
be a commutative ring with unit, denote by $R_{X}$ the constant sheaf of rings
on $X$ defined by $R$, and let ${\mathcal O}_{X}$ be a sheaf of rings on $X$. We
further assume that there is a morphism of sheaves of rings $i:R_{X}
\longrightarrow {\mathcal O}_{X}$, i.e., ${\mathcal O}_{X}$ is an
$R_{X}$-algebra. Let $\pi_{1}(X)$ denote the fundamental group of $X$ (for some
base point), and let $\Sigma$ be an arbitrary group, together with a group
homomorphism $\alpha:\pi_{1}(X) \longrightarrow \Sigma$. We are going to
describe a functor
$$
S: R\Sigma \textrm{-}{mod} \longrightarrow {\mathcal O}_{X} \textrm{-}{mod},
$$
from the category of modules over the group ring $R\Sigma$ to the category of
${\mathcal O}_{X}$-modules on $X$. 

\begin{con}\label{functor}
The group homomorphism $\alpha$ induces a morphism of group rings
$\tilde{\alpha}:R\pi_{1}(X) \longrightarrow R\Sigma$. If $\rho $ is an $R\Sigma$
module, then the change of rings functor $\tilde{\alpha}^{*}$ maps $\rho $ to an
$R\pi_{1}(X)$-module $\tilde{\alpha}^{*}(\rho)$. The Riemann Hilbert functor
$RH$ \cite[Theorem 2.5.15]{sza} maps the $R\pi_{1}(X)$-module
$\tilde{\alpha}^{*}(\rho)$ to a locally constant sheaf
$RH(\tilde{\alpha}^{*}(\rho))$ of $R$-modules on $X$. 
Finally, sending locally constant sheaves of $R$-modules to
their tensor product over $i$ with ${\mathcal O}_{X}$, i.e., $M
\mapsto M \otimes _{R _{X}}  {\mathcal O}_{X}$, defines a functor $i^{*}$ from
the category of locally constant
sheaves of
$R$-modules to the category of ${\mathcal O}_{X}$-modules.

In this way, we can define the functor $S$ as the composition of these three
functors,
so that, for every $R\Sigma$-module $\rho$, we have a well defined ${\mathcal
O}_{X}$-\
module:
\[
S(\rho):=i^{*} \circ RH\circ \tilde{\alpha}^{*}(\rho).
\]
\end{con}

\begin{Def}
The functor defined in Construction \ref{functor} above is called a Schottky
functor 
when $\Sigma$ is a free group and $\alpha:\pi_{1}(X) \longrightarrow \Sigma$ is 
surjective. We usually also assume that the ring $R$ is the complex field, 
$X$ is a complex manifold and in this case ${\mathcal O}_{X}$ is the sheaf
of holomorphic functions on $X$.
\end{Def}

\begin{prop}\label{prop_schottky_is_exact}
If $R=\C$ and $X$ is a complex manifold, then the Schottky functor $S$ is
faithful,
exact, additive, and compatible with direct sums and tensor
products.
\end{prop}
\begin{proof}
It suffices to show these properties (if they apply) for each of the functors
$\tilde{\alpha}^{*}$, $RH$ and $i^{*}$. The change of rings functor
$\tilde{\alpha}^{*}$ is additive and exact by \cite[CH 8. Proposition
8.33]{rot}, faithful by proof of \cite[CH 8. Proposition 8.33]{rot}, and it is
easy to see that it is compatible with direct sums and tensor products. The
equivalence of categories functor $RH$ is additive, exact, and compatible with
direct sums and tensor products because the inverse functor $RH^{-1}$ is induced
by the fiber functor mapping a sheaf ${\mathcal F}$ to its stalk ${\mathcal
F}_{x}$ at $x$ \cite[Theorem 2.5.14]{sza}, and this functor satisfies all these
properties. $RH$ is obviously faithful as well since $RH$ is an equivalence.
Finally, it can be checked directly that the functor $i^{*}$ is
additive and
compatible with direct sums and tensor products. Exactness and faithfulness
can be seen by looking at the stalks.
\end{proof}

As mentioned in the introduction, two cases are especially important here. Let
$\Sigma$ be a free group on $g$ generators.
The classical case is when $X$ is a Riemann surface of genus $g$
and $\alpha$ is defined as before Theorem \ref{thm_main_for_riemann_surfaces}. 
The other case is when
$X$ is a complex torus of dimension $g$ and $\alpha$ is as given just before
the statement of Theorem \ref{thm_main_for_tori}

\section{Schottky functors and
cohomology}\label{sec_schottky_functors_and_cohomology}

Let us assume now that $\Sigma$ is an arbitrary group and that $X$ is a complex
manifold, so that $S$ is additive and exact. In both abelian categories, $\C
\Sigma\textrm{-mod}$ and ${\mathcal
O}_{X}\textrm{-mod}$, one can compute \emph{Ext} groups using injective
resolutions. We will investigate the Schottky functor in this context.

Let $A$ and $B$ be two $\C\Sigma$-modules, let $\mathbf{I}$ denote an injective
resolution of $B$ and let $\mathbf{J}$ denote an
injective
resolution of $S(B)$. Then, since $\mathbf{J}$ is injective and $S(\mathbf{I})$
is an exact sequence, the identity
$f:=id_{S(B)}$ can be extended to a chain map $F:S(\mathbf{I})
\rightarrow
\mathbf{J}$ \cite[Theorem 6.16]{rot}. 
Then, after applying $Hom$, we obtain the following map of chain complexes
$$
Hom_{\C \Sigma}(A,\mathbf{I}) \overset{S_{*}}{\rTo} Hom_{{\mathcal
O}_{X}}(S(A),S(\mathbf{I})) \overset{F_{*}}{\rTo} Hom_{{\mathcal
O}_{X}}(S(A),\mathbf{J}).
$$
\begin{prop}\label{prop_naturality_of_S_and_cohomology}
The previous construction defines a morphism $S_{*}:= S_{*} \circ F_{*}$ of
graded $\C$-vector spaces
$$
S_{*}: Ext^{*}_{\C \Sigma}(A,B) \rightarrow Ext^{*}_{{\mathcal
O}_{X}}(S(A),S(B))
$$
that is natural in the following sense:
Let
$$
0 \rightarrow B' \overset{i}{\rightarrow} B \overset{p}{\rightarrow} B''
\rightarrow 0
$$
be an exact sequence of $\C\Sigma$-modules. Then there is a commutative diagram
of $\C$-vector spaces with exact rows,
\begin{diagram}[width=2pt,height=2em]
Ext_{\C\Sigma}^{n}(A,B') & \rTo^{i_{*}} &
Ext_{\C\Sigma}^{n}(A,B) & \rTo^{p_{*}} & Ext_{\C\Sigma}^{n}(A,B'') &
\rTo^{\delta} & Ext_{\C\Sigma}^{n+1}(A,B') \\
 \dTo^{S_{*}} &  & \dTo^{S_{*}} &  & \dTo^{S_{*}} & & \dTo^{S_{*}}  \\
 Ext_{{\mathcal O}_{X}}^{n}(S(A),S(B')) & \rTo^{S(i)_{*}} &
Ext_{{\mathcal O}_{X}}^{n}(S(A),S(B)) & \rTo^{S(p)_{*}} & Ext_{{\mathcal
O}_{X}}^{n}(S(A),S(B'')) & \rTo^{\delta} & Ext_{{\mathcal
O}_{X}}^{n+1}(S(A),S(B')). \\
\end{diagram}
\end{prop}
\begin{proof}
Using the injective version of the Horseshoe lemma \cite[Proposition
6.24]{rot} we may
find injective resolutions $\mathbf{I'}$, $\mathbf{I}$ and $\mathbf{I''}$ of
$B'$, $B$ and $B''$, respectively, and lifts $\tilde{i}$ and $\tilde{p}$ of $i$
and
$p$, respectively, such that they fit into a normal exact sequence (cf. \cite[V,
Section 2]{caei})
$$
0 \rightarrow \mathbf{I'} \overset{\tilde{i}}{\rightarrow} \mathbf{I}
\overset{\tilde{p}}{\rightarrow} \mathbf{I''} \rightarrow 0.
$$
We note that applying $S$ to this normal exact sequence of complexes again gives
normal exact sequences. Using again the Horseshoe lemma we may
find injective resolutions $\mathbf{J'}$, $\mathbf{J}$ and $\mathbf{J''}$ of
$S(B')$, $S(B)$ and $S(B'')$, respectively, and lifts $\widetilde{S(i)}$ and
$\widetilde{S(p)}$ of $S(i)$ and $S(p)$, respectively, such that they fit into
an
exact sequence 
$$
0 \rightarrow \mathbf{J'} \overset{\widetilde{S(i)}}{\rightarrow} \mathbf{J}
\overset{\widetilde{S(p)}}{\rightarrow} \mathbf{J''} \rightarrow 0.
$$ 
As above, we may also find chain maps $F':S(\mathbf{I'}) \rightarrow
\mathbf{J'}$ and $F'':S(\mathbf{I''}) \rightarrow \mathbf{J''}$, lifting the
identities $f'=id_{S(B')}$ and $f''=id_{S(B'')}$, respectively. Furthermore,
since $\mathbf{J'}$ is injective and $S(\mathbf{I''})$ is acyclic, we
can apply the injective version of 
\cite[V \S 2 Proposition 2.3]{caei} and
find a chain map 
$F:S(\mathbf{I}) \rightarrow
\mathbf{J}$ above the identity map $f=id_{S(B)}$, in such a way that there is a
commutative diagram of complexes
\begin{diagram}[width=2pt,height=2em]
0 & \rTo & S(\mathbf{I'}) & \rTo^{S(\tilde{i})}  & S(\mathbf{I}) &
\rTo^{S(\tilde{p})} &
S(\mathbf{I''}) & \rTo &0 \\
 &&\dTo ^{F'}&   & \dTo^{F}  &   & \dTo^{F''}  &  &   &   &   \\
0 & \rTo & \mathbf{J'} & \rTo^{\widetilde{S(i)}} 
& \mathbf{J} & \rTo^{\widetilde{S(p)}} & \mathbf{J''} & \rTo
&0. \\
\end{diagram}

After applying $Hom$ and adding the initial sequence, we obtain the following
commutative diagram
\begin{diagram}[width=2pt,height=2em]
0 & \rTo & Hom_{\C \Sigma}(A,\mathbf{I'}) & \rTo^{\tilde{i}_{*}}  & Hom_{\C
\Sigma}(A,\mathbf{I}) & \rTo^{\tilde{p}_{*}} &
Hom_{\C \Sigma}(A,\mathbf{I''}) & \rTo &0 \\
 &&\dTo ^{S_{*}}&   & \dTo^{S_{*}}  &   & \dTo^{S_{*}}  &  &   &   &   \\
0 & \rTo & Hom_{{\mathcal O}_{X}}(S(A),S(\mathbf{I'})) &
\rTo^{S(\tilde{i})_{*}} 
& Hom_{{\mathcal
O}_{X}}(S(A),S(\mathbf{I})) & \rTo^{S(\tilde{p})_{*}} & Hom_{{\mathcal
O}_{X}}(S(A),S(\mathbf{I''})) & \rTo 
&0 \\
 &&\dTo ^{\psi'_{*}}&   & \dTo^{\psi_{*}}  &   & \dTo^{\psi''_{*}}  &  &   &   &
  \\
0 & \rTo & Hom_{{\mathcal O}_{X}}(S(A),\mathbf{J'}) &
\rTo^{\widetilde{S(i)})_{*}} 
& Hom_{{\mathcal
O}_{X}}(S(A),\mathbf{J}) & \rTo^{\widetilde{S(p)})_{*}} & Hom_{{\mathcal
O}_{X}}(S(A),\mathbf{J''}) & \rTo
&0. \\
\end{diagram}
The top and the bottom sequences are exact, since $\mathbf{I}'$ and
$\mathbf{J}'$ are sequences of
injective objects. Now we can omit the horizontal sequence in the middle and and
consider the
vertical composite maps. Then we can apply \cite[Theorem 6.13]{rot} which shows
the claim.
\end{proof}

Consider now an extension of $A$ by $B$ in the category of $\C\Sigma$-modules,
i.e., an exact sequence
\begin{diagram}[width=2pt,height=2em]
0 & \rTo & B  & \rTo & E & \rTo & A & \rTo & 0. \\
\end{diagram}
Such extensions are classified by the \emph{Yoneda Ext} groups (cf. \cite[Ch.
7.2]{rot}
or \cite[Ch VII]{mit}) that we will denoted by
$YExt^{1}(A,B)$. Let us recall the comparison isomorphism (cf. \cite[Ch 7.2
Lemma
7.27 and Theorem 7.35]{rot} or the
 injective version in \cite[Ch VII, Section 7]{mit}) 
between $Ext^{1}(A,B)$ and $YExt^{1}(A,B)$.  
Let $\mathbf{I}$
be an injective resolution of $B$. Then the identity $\phi=id_{B}$ can be lifted
to a chain map $\Phi$
\begin{diagram}[width=2pt,height=2em]
0 & \rTo & B  & \rTo & I^{0} & \rTo & I^{1}&  \rTo & I^{2} & \rTo & \ldots \\
 &  & \uTo^{id_{B}}  &  & \uTo^{\Phi_{0}} &  & \uTo ^{\Phi_{1}} &  &\uTo
^{\Phi_{2}} &&  \\
0 & \rTo & B  & \rTo & E & \rTo & A & \rTo & 0. & & \\
\end{diagram}
Because of the quadrant on the right the morphism $\Phi_{1}$ represents a
cocycle class $[\Phi_{1}]$ in
$Ext^{1}_{\C \Sigma}(A,B)$. An analogous reasoning applies also to extensions
in the category of ${\mathcal
O}_{X}$-modules. Moreover, the functor $S$ defines a canonical morphism
$YExt^{1}_{\C \Sigma}(A,B) \rightarrow YExt^{1}_{{\mathcal O}_{X}}(S(A),S(B))$
that
we will also denote by $S_{*}$. We have the following comparison result:
\begin{prop}\label{prop_compatibility_Yoneda_Derived_ext}
The following diagram is commutative:
\begin{diagram}[width=2pt,height=2em]
 YExt^{1}_{\C \Sigma}(A,B) & \rTo^{C_{\C \Sigma}}_{\simeq} & Ext^{1}_{\C
\Sigma}(A,B) \\
 \dTo ^{S_{*}} &  & \dTo ^{S_{*}} \\
 YExt^{1}_{{\mathcal O}_{X}}(S(A),S(B)) & \rTo^{C_{{\mathcal O}_{X}}}_{\simeq} &
Ext^{1}_{{\mathcal
O}_{X}}(S(A),S(B)) \\
\end{diagram}
where $C_{\C \Sigma}$ and $C_{{\mathcal O}_{X}}$ are the comparison
isomorphisms.
\end{prop}
\begin{proof}
 Let $0 \rightarrow B  \rightarrow E \rightarrow A \rightarrow 0 $ represent an
extension class $[E]$ in $YExt^{1}_{\C \Sigma}(A,B)$. 
Then $C_{\C \Sigma}$ maps $[E]$ to the class $[\Phi_{1}]$, and $S_{*}$ maps
$[\Phi_{1}]$
to the class $[ F_{1} \circ S(\Phi_{1}) ]$, where $F_{1} \circ S(\Phi_{1})$
 is the horizontal map in the diagram 
\begin{diagram}[width=2pt,height=2em]
0 & \rTo & S(B)  & \rTo & J^{0} & \rTo & J^{1}&  \rTo & \ldots \\
 &  & \uTo^{id_{S(B)}}  &  & \uTo^{F_{0}} &  & \uTo ^{F_{1}} &  &  \\
0 & \rTo & S(B)  & \rTo & S(I^{0}) & \rTo & S(I^{1}) & \rTo & \ldots\\
 &  & \uTo^{S(id_{B})}  &  & \uTo^{S(\Phi _{0})} &  & \uTo ^{S(\Phi_{1})} &  & 
\\
0 & \rTo & S(B)  & \rTo & S(E) & \rTo & S(A) & \rTo & 0. \\
\end{diagram}
On the other side $\psi =id_{S(B)}$ can be extended to a chain map $\Psi$ 
\begin{diagram}[width=2pt,height=2em]
0 & \rTo & S(B)  & \rTo & J^{0} & \rTo & J^{1}&  \rTo & \ldots \\
 &  & \uTo^{id_{S(B)}}  &  & \uTo^{\Psi_{0}} &  & \uTo ^{\Psi_{1}} &  &  \\
0 & \rTo & S(B)  & \rTo & S(E) & \rTo & S(A) & \rTo & 0, \\
\end{diagram}
and so the class $[E]$ under the map $C_{{\mathcal O}_{X}} \circ S_{*}$  is
mapped to the class $[\Psi_{1}]$.

Since $\Psi$ and $F \circ S(\Phi)$ are two chain maps over $id_{S(B)}$, and
because the
comparison isomorphism $C_{{\mathcal O}_{X}}$ is well defined, the two morphisms
$\Psi_{1}$ and $F_{1} \circ S(\Phi_{1})$
must represent the same
class in $Ext^{1}_{{\mathcal O}_{X}}(S(A),S(B))$. Therefore, the diagram
commutes.
\end{proof}

\begin{rem}
Although it is not necessary here, the results of this chapter can be stated in
more generality. We would like
to refer the interested reader to an earlier version (arxiv.org/abs/1102.3006v2,
Sections 2 and 3) of this preprint for further details.
\end{rem}

Finally, we calculate the dimensions of some Ext groups:

\begin{prop}\label{prop_explicit_computation_some_Ext}
 \begin{itemize}
\item [a)] If $X$ is a complex torus of dimension $g$, then
$Ext^{n}_{{\mathcal O}_{X}}({\mathcal O}_{X},{\mathcal O}_{X})$ is a $\C$-vector
space of dimension $1$, if $n=0$, and of dimension $g$, if $n=1$.
\item [b)] If $X$ is a compact Riemann surface of genus $g$,
then $Ext^{n}_{{\mathcal O}_{X}}({\mathcal O}_{X},{\mathcal O}_{X})$ is a
$\C$-vector space of dimension $1$, if $n=0$, and of dimension $g$, if $n=1$. For
any sheaf $F$ of ${\mathcal O}_{X}$-modules, the $\C$-vector spaces
$Ext^{n}_{{\mathcal O}_{X}}({\mathcal O}_{X},F)$ vanish for all $n\geq 2$.
\item[c)] If $\Sigma$ is a free abelian group of rank $g$, then $Ext^{n}_{\C
\Sigma}(\C,\C)$ is a $\C$-vector
space of dimension $1$, if $n=0$, and of dimension $g$, if $n=1$.
\item[d)] If $\Sigma$ is a free group of $g$ free generators, then
$Ext^{0}_{\C\Sigma}(\C,\C)$ is a free $\C$-module of rank $1$, and
$Ext^{1}_{\C\Sigma}(\C,\C)$ is a free $\C$-module of rank $g$. For any
$\C\Sigma$-module $M$ the $\C$-vector spaces
$Ext^{n}_{\C\Sigma}(\C,M)$ vanish for all $n\geq 2$.
\end{itemize}
\end{prop}
\begin{proof}
The assertions in a) and b) are well known (See e.g. \cite{gun} or \cite{bila}).
For c) and d) the claims for $n=0,1$ follow from \cite[X \S 4, (5) and
(6)]{caei}, and the second claim of d) follows from \cite[X \S 5]{caei}.  
\end{proof}

\section{Unipotent bundles and unipotent
representations}\label{section_unipotent_bd_group}

Let $R$ be a commutative ring with unit, let $\Sigma$ be a group, and let
$(X,{\mathcal O}_{X})$ be a ringed space (commutative with unit).
\begin{Def}
 \begin{itemize}
  \item[a)] A $R\Sigma$-module $M$ is called \emph{unipotent of level} $r$
($r \in \N$), if there exists a filtration of $R\Sigma$-modules $0=M_{0}\subset
\ldots \subset M_{r-1} \subset M_{r}=M$, such that $M_{i+1}/M_{i} \simeq R$ for
$i=0,\ldots , r-1$. We denote the collection of unipotent $R\Sigma$-modules of
level $r$ by ${\bf Un}_{R\Sigma}^{r}$ and by ${\bf Un}_{R\Sigma}$ its union over
all $r\geq 0$. These are full subcategories of the category of
$R\Sigma$-modules.
\item[b)] A ${\mathcal O}_{X}$-module $F$ on $X$ is called \emph{unipotent of
level} $r$ ($r \in \N$), if there exists a filtration of sub ${\mathcal
O}_{X}$-modules $0=F_{0}\subset \ldots \subset F_{r-1} \subset F_{r}=F$, such
that $F_{i+1}/F_{i} \simeq {\mathcal O}_{X}$ for $i=0,\ldots , r-1$. We denote
collection of unipotent ${\mathcal O}_{X}$-modules of level $r$ by ${\bf
Un}_{{\mathcal O}_{X}}^{r}$ and ${\bf Un}_{{\mathcal O}_{X}}$ its union over all
$r\geq 0$. These are full subcategories of the category of ${\mathcal
O}_{X}$-modules.
 \end{itemize}
\end{Def}

\begin{rem}\label{rem:description_unipotent}
If $M$ is a unipotent $R\Sigma$-module of level $r$, then $M$ is also a free
$R$-module of rank $r$. If $F$ is a unipotent ${\mathcal O}_{X}$-module of level
$r$ on $X$, then $F$ is a locally free ${\mathcal O}_{X}$-module of rank $r$
(also called vector bundle). Both properties follow by induction using \cite[4 I
Ch 0 (5.4.9)]{ega}. 
\end{rem}

\begin{rem}\label{rem_Lekause_unipotent} If $M \in {\bf Un}_{R\Sigma}^{r}$, then
after choosing a basis of $M$, we may interpret $M$ as a representation
$\rho:\Sigma \rightarrow GL_{r}(R)$. If $R=\C$ is the field of complex numbers,
then by Kolchins's Theorem \cite[Part I, Chapter V]{ser} one may choose the
basis of $M$ in such a way that for all $\sigma \in \Sigma$, the matrix
$\rho(\sigma)$ is upper triangular with ones on the diagonal.
In this case one can show \cite[proof of Proposition 11.4]{lek01} that there is
a unipotent $\C \Sigma$-module $M_{r-1}$ of level $r-1$ and an exact sequence
\begin{equation}\label{eq_alternative_desc_unipotentrep}
 0\rightarrow \C \rightarrow M \rightarrow M_{r-1} \rightarrow 0.
\end{equation}
This alternative description of unipotent representations will be used later.
\end{rem}

In \cite{sim}, an important correspondence between flat bundles on compact
K\"ahler manifolds 
and certain classes of Higgs bundles is obtained.
When considering unipotent vector bundles, this result has the following
consequence.

\begin{prop}\label{rem_Unipotent_bundles_are_flat}(Simpson,\cite{sim})
Let $Y$ be a complex compact K\"ahler manifold. Then every
unipotent vector bundle $U$ on $Y$ admits a flat holomorphic connection.
\end{prop}
\begin{proof} 
Since $U$ is a successive extension of the trivial bundle ${\mathcal O}_{Y}$, it
follows by induction that all the rational characteristic classes of $U$ of
positive
degree 
vanish. Moreover, $U$ is also a successive extension of the stable vector
bundle ${\mathcal O}_{Y}$. Therefore, it follows from \cite[Section 3 -
Examples, remarks following Lemma 3.5, p. 36/37]{sim} 
that $U$ admits a flat holomorphic connection.
\end{proof}

\begin{rem}
Note also that any unipotent vector bundle over a compact K\"ahler manifold is
always semi-stable.
\end{rem}

\section{The case of complex tori}\label{section:CaseOfComplexTori}

As in the introduction let $X=V/\Lambda$ be a complex torus of dimension $g$, 
where $\Lambda$ is canonically identified with $\pi_{1}(X)$.

\begin{lem}\label{lem:unipotentBundlesareFlat}
Let $U$ be a unipotent ${\mathcal O}_{X}$-module of level $r$ on $X$. Then there
exists a unipotent $\C \Lambda$-module $\rho$ of level $r$, such that $E_{\rho}$
is isomorphic to $U$.
\end{lem}
\begin{proof}
Since a complex torus is a compact K\"ahler manifold,
by Proposition \ref{rem_Unipotent_bundles_are_flat} the vector bundle
$U$ over $X$ admits a flat holomorphic connection.
So, there exists a $\C \Lambda$-module $\rho$, such that $E_{\rho} \simeq
U$, and we need to show that one can choose such a module that is unipotent. We
claim that each indecomposable component of $U$ is also unipotent: By
\cite[Th\'eor\`eme 3]{mat}, each indecomposable component of $U$ is isomorphic
to a vector bundle of the form $L_{c} \otimes U_{c}$, where $L_{c}$ is a line
bundle of degree zero, and $U_{c}$ is an indecomposable unipotent vector bundle.
If we have such a (nonzero) component, then we obtain an injection $L_{c}
\hookrightarrow U$, and since $U$ is unipotent, we can form the following
commutative diagram
\begin{diagram}[width=2pt,height=1.5em]
0 & \rTo & U_{r-1} & \rTo & U & \rTo & {\mathcal O}_{X} & \rTo & 0, \\
  &   &   &   & \uTo  & \ruTo &   &   &   \\
  &   &   &   & L_{c} &  &  &  & \\
\end{diagram}
where $U_{r-1}$ is unipotent of level $r-1$, and the sequence is exact. If
$L_{c}$ is nontrivial, then the diagonal map must be zero since $Hom _{{\mathcal
O}_{X}}( L_{c},{\mathcal O}_{X}) $ would vanish in this case. Therefore the
inclusion $L_{c} \hookrightarrow U$ must factor over the kernel $U_{r-1}$, so we
obtain a new inclusion $L_{c} \hookrightarrow U_{r-1}$. By inverse induction we
could find then an injection $L_{c} \hookrightarrow 0$ what is clearly a
contradiction. This shows that each indecomposable component of $U$ is
unipotent, and henceforth we may assume that $U$ is indecomposable. In this case
the $\C\Lambda$ module $\rho$ is indecomposable as well (use that $S$ is
compatible with direct sums). Since $\Lambda$ is abelian, such an indecomposable
$\C\Lambda$-module must be isomorphic to the tensor product $\alpha \otimes
\beta$, where $\alpha$ is a $\C\Lambda$-module of $\C$-dimension one, and
$\beta$ is a unipotent $\C\Lambda$-module. (This can be seen analogously as in
\cite[Corollary 2.7]{luda}) By the same reasoning as above we see that the line
bundle $E_{\alpha}$ must be    trivial since $U \simeq E_{\alpha} \otimes
E_{\beta}$ is indecomposable and unipotent. Therefore we may replace $\rho$ by
$\alpha^{-1}\otimes \rho = \beta $, and so we have found a $\rho$ that is
unipotent and satisfies $E_{\rho} \simeq U$, this shows our claim. 
\end{proof}

Consider now the morphism $\alpha : \pi_{1}(X)=\Lambda \longrightarrow \Sigma$ 
defined just before Theorem \ref{thm_main_for_tori}.
Using \cite[Lemme 5.1]{mor} we can show the following proposition:

\begin{prop}
\label{pro:unipotent-Schottky} The Schottky functor $S$ is essentially
surjective for unipotent objects, i.e., for each unipotent ${\mathcal
O}_{X}$-module $U$ on $X$ of level $r$, there exists a unipotent $\C
\Sigma$-module $\rho$ of level $r$, such that $E_{\rho}=S(\rho)$ is isomorphic
to $U$.
\end{prop}
\begin{proof}
It follows from Lemma \ref{lem:unipotentBundlesareFlat}, that there exists a
unipotent representation $\rho:\Lambda\to  GL(W)$, such that $E_{\rho}$ is
isomorphic to $U$, and we need to find such a representation that factors over
$\Sigma$.
Denote by $Un(W)$ the unipotent subgroup of $GL(W)$. Let $H$ be
the smallest commutative Lie group containing $\rho(\Lambda)$ inside
$Un(W)$, so that:\[
\rho(\Lambda)\subset H\subset Un(W).\]
Note that any holomorphic function $f:V\to H$, due to the commutativity
of $H$, verifies\[
f(\lambda + z) \cdot \rho(\lambda) \cdot f(z)^{-1}=f(\lambda + z) \cdot
f(z)^{-1} \cdot\rho(\lambda).\]
Also because $H$ is commutative, there is a well defined notion of
logarithm: $\log:H\to\mathfrak{h}$ where $\mathfrak{h}$ is the Lie
algebra of $H$. Let $A_{j}:=-\log(\rho(\lambda_{g+j}))\in\mathfrak{h}$,
$j=1,...,g$, for the generators $\lambda_{g+j} \in I \cdot \Z^{g} \subset
\Lambda$, and consider
the $\mathfrak{h}$-valued 1-form $\omega=A_{1}dz_{1}+\cdots+A_{g}dz_{g}$
where $z_{1},...,z_{g}$ are coordinates on $V$ dual to the basis
$e_{1},...,e_{g}$. 
Let $f(z):=\exp(\int_{0}^{z}\omega)$. One easily checks that this
is well defined and belongs to $GL(W)$. Then $f(\lambda + z) \cdot
f(z)^{-1}=\exp(\int_{z}^{\lambda+z}\omega)$
and this expression is independent of $z$ as $\omega$ has constant coefficients.
So, for $\lambda\in I \cdot \Z^{g}$ we can uniquely write
$\lambda=c_{1}e_{1}+\cdots+c_{g}e_{g}$
with $c_{j}\in\mathbb{Z}$ and we have\begin{alignat*}{1}
f(\lambda + z)\cdot f(z)^{-1} & =\exp(\int_{0}^{\lambda}\omega)\\
 & =\exp(\int_{0}^{\lambda}A_{1}dz_{1}+\cdots+A_{g}dz_{g})\\
 & =\exp(c_{1}A_{1}+\cdots+c_{g}A_{g})\\
 & =\exp(A_{1})^{c_{1}}\cdots\exp(A_{g})^{c_{g}}\\
 & =\rho(e_{1})^{-c_{1}}\cdots\rho(e_{g})^{-c_{g}}\\
 & =\rho(c_{1}e_{1}+\cdots+c_{g}e_{g})^{-1}\\
 & =\rho(\lambda)^{-1},\end{alignat*}
where we used $\int_{e_{j}}dz_{k}=\delta_{jk}$. This means that, defining \[
\sigma(\lambda)=f(\lambda + z) \cdot f(z)^{-1} \cdot\rho(\lambda),\]
we have shown that $\sigma$ is a Schottky representation w.r.t. $\alpha$,
because $\sigma(\lambda)=1$ for all $\lambda\in I \cdot \Z^{g}$, i.e., $\sigma$
factors over $\Sigma$. Moreover, 
$\rho$ and $\sigma$ define isomorphic vector bundles, by \cite[Lemme 5.1]{mor}.
\end{proof}

\begin{lem}\label{lem_yoneda_ext1_S_iso_CT}
The map
$$
S_{*} :YExt^{1}_{\C \Sigma}(\C,\C) \longrightarrow YExt^{1}_{{\mathcal
O}_{X}}({\mathcal O}_{X},{\mathcal O}_{X})
$$
is an isomorphism of $g$-dimensional $\C$-vector spaces.
\end{lem}
\begin{proof}
Both spaces are $\C$-vector spaces of dimension $g$ by Proposition
\ref{prop_explicit_computation_some_Ext} a) and c), and because of the
comparison
isomorphisms $C_{\C \Sigma}$ and $C_{{\mathcal O}_{X}}$ from Proposition
\ref{prop_compatibility_Yoneda_Derived_ext}. So, it suffices to show
that $S_{*}$ is
injective. Let $Y:=(\C \hookrightarrow \rho \twoheadrightarrow \C)$ represent an
extension class $[Y]$ in $YExt^{1}_{\C \Sigma}(\C,\C)$, and assume that
$S_{*}([Y])$ is
trivial. Then $S_{*}(Y)$ splits and so $S(\rho)$ is isomorphic to the trivial
rank two ${\mathcal O}_{X}$-module.
After choosing a basis, we may assume that
$\rho$ is a
representation into $GL_{2}(\C)$ and given in matrix form by
\begin{equation}\label{eq_explicit_rk2_unipotent_torus}
\rho(\lambda) =
\left( \begin{array}{c  c }
1 & \rho '(\lambda)  \\
0 & 1 \\
\end{array} \right),
\end{equation}
for a group homomorphism $\rho ': \Lambda \longrightarrow \C$. It follows then from \cite[Lemme 6.3]{mat} that $\rho' \in
Hom_{\C}(V,\C)=Hom_{\Z}(I \cdot \Z^{g},\C)$ because $S(\rho)$ is decomposable.
But by definition of $\alpha$, we must have $\rho'\in Hom_{\C}(Z \cdot
\Z^{g},\C)$ so $\rho'$ and $\rho$ must be trivial. Finally, since $\rho$ is
trivial, $Y$ may be considered as an exact sequence of vector spaces, and such a
sequences splits, so $Y$ must represent the trivial extension class.
\end{proof}

We are now able to prove \textbf{Theorem \ref{thm_main_for_tori}}:

\begin{proof}By Proposition \ref{pro:unipotent-Schottky} we know that all
unipotent vector bundles lie in the essential image of the functor $S$.
Moreover,
since $\C$ is a field, we know by Proposition \ref{prop_schottky_is_exact} c)
that $S$ is also faithful. So, it remains to show that $S$ is full as well. If
$M_{1},M_{2}$ are two unipotent $\C\Sigma$-modules, then $S_{*}$ defines an
injective morphism of $\C$-vector spaces $Hom_{\C \Sigma}(M_{1},M_{2})
\hookrightarrow Hom_{{\mathcal O}_{X}}(S(M_{1}),S(M_{2}))$, since $S$ is
faithful. So, it suffices to show that both have the same dimension, or
equivalently that $Hom_{\C \Sigma}(\C,M_{1}^{*}\otimes M_{2})$ and
$Hom_{{\mathcal O}_{X}}({\mathcal O}_{X},S(M_{1})^{*}\otimes S(M_{2}))$ have the
same dimension ($^{*}$ denotes dual). It is also easy to see that
$M_{1}^{*}\otimes M_{2}$ is unipotent. So, our claim follows if we can show that
if $U_{r}$ is a unipotent $\C\Sigma$-module, then $S_{*}:Hom_{\C
\Sigma}(\C,U_{r}) \rightarrow Hom_{{\mathcal O}_{X}}(S(\C),S(U_{r}))$ is an
isomorphism. We will show this by induction on the level. For $r=0,1$ this is
trivial, and so let us assume that this is known for all unipotent bundles of
level less than $r$. By Section \ref{section_unipotent_bd_group}, Equation
\ref{eq_alternative_desc_unipotentrep}, there exists a unipotent
$\C\Sigma$-module or level $r-1$ and an exact sequence
\begin{equation}
 0\rightarrow \C \rightarrow U_{r} \rightarrow U_{r-1} \rightarrow 0.
\end{equation}
Moreover by Proposition \ref{prop_naturality_of_S_and_cohomology} we have the
following
commutative diagram in which the horizontal sequences are exact:
\begin{diagram}[width=2pt,height=2em]
0 & \rTo & Hom_{\C\Sigma}(\C,\C) & \rTo  & Hom_{\C\Sigma}(\C,U_{r}) & \rTo  &
Hom_{\C\Sigma}(\C,U_{r-1}) & \rTo^{\delta} & Ext_{\C\Sigma}^{1}(\C,\C) 
\\
 \dTo &  & \dTo^{S_{*}} &  & \dTo^{S_{*}} &  & \dTo^{S_{*}} & & \dTo^{S_{*}}   
\\
0 & \rTo & Hom_{{\mathcal O}_{X}}(S(\C),S(\C)) & \rTo& Hom_{{\mathcal
O}_{X}}(S(\C),S(U_{r})) & \rTo & Hom_{{\mathcal O}_{X}}(S(\C),S(U_{r-1})) &
\rTo^{\delta } & Ext_{{\mathcal O}_{X}}^{1}(S(\C),S(\C))  \\
\end{diagram}
The second
and the fourth vertical arrows are isomorphisms by the induction hypothesis, and
it follows from Lemma \ref{lem_yoneda_ext1_S_iso_CT} and Proposition
\ref{prop_compatibility_Yoneda_Derived_ext} that the fifth vertical
arrow is an isomorphism. It follows then by the five lemma that the middle arrow
is an isomorphism as well, and our claim follows by induction. This shows that
$S$ is full as well, and so by abstract category it follows that $S$ induces an
equivalence.
\end{proof}

\begin{rem}
\begin{itemize}
\item[a)] If $X$ is an elliptic curve, then it was shown that in \cite[Lemma
5]{flo} that all unipotent bundles lie in the essential image of the Schottky
functor $S$. In the rank two case, an analogous result was already shown in
\cite{mat} Section 6. 
\item[b)] In the $p$-adic case, an analogous result was shown in \cite[Lemma
4.3]{ludb} by a different method.
\end{itemize}
\end{rem}

\section{The case of Riemann Surfaces}\label{section_RiemannSurface}

As in the introduction let $X$ be a compact Riemann Surface of genus $g$, and
consider the morphism $\alpha:\pi_{1}(X) \rightarrow \Sigma$. 

We will need the following lemma:
\begin{lem}\label{lem_unipotent_rk2_S_injective}
Let $\rho$ be a unipotent $\C \pi_{1}(X)$-module of rank $2$, and assume that
$\rho$ is induced by a $\C \Sigma$-module. If $S(\rho)$ is isomorphic to the
trivial rank two vector bundle on $X$, then $\rho$ is trivial.
\end{lem}
\begin{proof}
(The proof is analogous to \cite[Lemme 6.3]{mat}) We may assume that $\rho$ is a
representation into $GL_{2}(\C)$ and given in matrix form by
\begin{equation}\label{eq_choice_function_RS}
\rho(\gamma) =
\left( \begin{array}{c  c }
1 & \rho '(\gamma)  \\
0 & 1 \\
\end{array} \right),
\end{equation}
for a group homomorphism $\rho ': \Sigma \longrightarrow \C$. Then since
$S(\rho)$ is trivial, we can apply \cite[4 Lemma 2]{flo}, and so there exists a
holomorphic map $f$ from the universal covering $\widetilde{X}$ of $X$ into
$GL_{2}(\C)$, such that $f(\gamma \cdot x) =\rho(\gamma) \cdot f(x) $. We may
fix a point $\widetilde{x_{0}}$ above a base point $x_{0}$ of $X$, and replace
$f(x)$ by $f(x) \cdot f(\widetilde{x_{0}})^{-1}$ so that we can assume
$f(\widetilde{x_{0}})=I$. Let us set
\begin{equation}\label{eq_difference_between_VB}
f(x) =
\left( \begin{array}{c  c }
u(x) & v(x)  \\
w(x) & z(x) \\
\end{array} \right).
\end{equation}
From the equation $f(\gamma \cdot x) =\rho(\gamma) \cdot f(x) $, one can deduce
the following equations: 
$w(\gamma \cdot x) = w(x)$ and $z(\gamma \cdot x) = z(x)$. Therefore $w$ and $z$
are invariant under the $\pi_{1}(X)$-action, and so define global holomorphic
sections of $X$, and so they must be constant. From our normalization condition
we see that $w=0$ and $z=1$. This implies also the equation $u(\gamma \cdot
x)=u(x)$, so by the same reasoning as before, we can see that $u=1$. From these
computations we obtain the equation $v(\gamma \cdot x) =v(x) + \rho '(\gamma)$.
Therefore, the holomorphic differential $dv$ is invariant under the
$\pi_{1}(X)$-action, and so it descends to a holomorphic differential $\omega$
on $X$. We have the following equation:
$$
\int_{\gamma} \omega = \int_{\widetilde{x_{0}}}^{\gamma \widetilde{x_{0}}} dv=
v(\gamma \widetilde{x_{0}}) - v(\widetilde{x_{0}}) = \rho '(\gamma). 
$$
By our initial assumptions on $\rho$, the homomorphism $\rho '$ must vanish for
the $a$-periods $a_{1},\ldots ,a_{g}$. On the other side, there exists a basis
$\zeta_{1},\ldots ,\zeta_{g}$ of holomorphic differentials on $X$ such that
$\int_{a_{i}}\zeta_{k}=\delta_{ik}$ \cite[III 2.8]{fakr}. This implies that
$\omega$ must be the trivial differential, and so $\rho '$ and $\rho$ must be
trivial.
\end{proof}

\begin{lem}\label{lem_yoneda_ext1_S_iso_RS}
The map
$$
S_{*}: YExt^{1}_{\C \Sigma}(\C,\C) \longrightarrow YExt^{1}_{{\mathcal
O}_{X}}({\mathcal O}_{X},{\mathcal O}_{X})
$$
is an isomorphism of $g$-dimensional $\C$-vector spaces.
\end{lem}
\begin{proof}
Both $\C$-vector spaces are of dimension $g$ by Proposition
\ref{prop_explicit_computation_some_Ext} b) and d) and the comparison
isomorphisms $C_{\C \Sigma}$ and $C_{{\mathcal O}_{X}}$ from Proposition
\ref{prop_compatibility_Yoneda_Derived_ext}.
So it
suffices to show that $S_{*}$ is
injective. The proof is now analogous to Lemma \ref{lem_yoneda_ext1_S_iso_CT}
with the difference that we can apply Lemma \ref{lem_unipotent_rk2_S_injective}
instead of \cite[Lemme 6.3]{mat}.
\end{proof}

We are now able to show \textbf{Theorem \ref{thm_main_for_riemann_surfaces}}:
\begin{proof}
Using the same arguments as in the proof of Theorem \ref{thm_main_for_tori}, we
can deduce that $S$ is fully faithful (for unipotent representations). Let us
now show using induction that $S$ is essentially surjective. We first are going
to show that the map $S_{*}:Ext_{\C\Sigma}^{1}(\C,U) \longrightarrow
Ext_{{\mathcal
O}_{X}}^{1}(S(\C),S(U))$ is an isomorphism for all unipotent $\C\Sigma$-modules
$U$. For level $1$, i.e., $U=\C$ this follows from Lemma
\ref{lem_yoneda_ext1_S_iso_RS} and Proposition
\ref{prop_compatibility_Yoneda_Derived_ext}. So, let us assume that this map is
an isomorphism
for all unipotent $\C\Sigma$-modules of level less than $r$, and let us fix a
unipotent $\C\Sigma$-module $U_{r}$ of rank $r$. Then since $U_{r}$ is
unipotent, there exists a unipotent $\C\Sigma$-module $U_{r-1}$ of level $r-1$
and an exact sequence
\begin{equation}
 0\rightarrow  U_{r-1} \rightarrow U_{r} \rightarrow \C \rightarrow 0.
\end{equation}
Moreover, by Propositions \ref{prop_naturality_of_S_and_cohomology} and
\ref{prop_compatibility_Yoneda_Derived_ext}
we have the following
commutative diagram:
\begin{diagram}[width=2pt,height=2em]
 Hom(\C,\C) & \rTo^{\delta} & Ext^{1}(\C,U_{r-1}) & \rTo &
Ext^{1}(\C,U_{r})
&\rTo & Ext^{1}(\C,\C) & \rTo& 0\\
   \dTo^{S_{*}} &   & \dTo^{S_{*}} &   & \dTo^{S_{*}} &  & \dTo^{S_{*}} &
&\dTo^{S_{*}}\\
 Hom(S(\C),S(\C)) & \rTo^{\delta} & Ext^{1}(S(\C),S(U_{r-1})) & \rTo &
Ext^{1}(S(\C),S(U_{r})) &\rTo & Ext^{1}(S(\C),S(\C)) & \rTo &0\\
\end{diagram}
(The two terms $Ext^{2}(\C,U_{r-1})$ and $Ext^{2}(S(\C),S(U_{r-1})$ on the right
hand side vanish, because of Proposition
\ref{prop_explicit_computation_some_Ext} b) and d)). It is clear that the first
vertical arrow is
an isomorphism, and the second and fourth vertical arrows are isomorphisms by
induction hypothesis. Now, it follows from the five lemma that the third vertical
arrow is an isomorphism, so our claim follows by induction. We can now show that
$S$ is essentially surjective. This is known in the level one case, so let us
assume that all unipotent vector bundles of level less than $r$ are in the
essential image of $S$. Let us fix a unipotent vector bundle $G_{r}$ of level
$r$. Then since $G_{r}$ is unipotent, there exists a unipotent vector bundle
$G_{r-1}$ of level $r-1$ and an exact sequence
\begin{equation}
 0\rightarrow  G_{r-1} \rightarrow G_{r} \rightarrow {\mathcal O}_{X}
\rightarrow 0,
\end{equation}
defining a class in $YExt^{1}({\mathcal O}_{X},G_{r-1})$. By induction
hypothesis
there exists an unipotent representation $U_{r-1}$ with $S(U_{r-1})=G_{r-1}$.
Then by Proposition
\ref{prop_compatibility_Yoneda_Derived_ext} the morphism
$S_{*}:YExt^{1}(\C,U_{r-1}) \rightarrow YExt^{1}({\mathcal O}_{X},G_{r-1})$ is
surjective and so there exists an exact sequence of unipotent
$\C\Sigma$-modules 
\begin{equation}
 0\rightarrow  U_{r-1} \rightarrow U_{r} \rightarrow \C \rightarrow 0,
\end{equation}
and a commutative diagram

\begin{diagram}[width=2pt,height=2em]
0 & \rTo & S(U_{r-1}) & \rTo & S(U_{r}) & \rTo & S(\C) &\rTo & 0\\
  & &     \dTo^{=} &   & \dTo^{f} &   & \dTo^{=} &  \\
0 &\rTo & G_{r-1} & \rTo & G_{r} & \rTo & {\mathcal O}_{X}  &\rTo & 0.\\
\end{diagram}
By the five lemma the middle arrow must be an isomorphism so $G_{r}$ lies in the
essential image of $S$, and by induction it follows that $S$ is essentially
surjective. Finally, since $S$ is fully faithful and essentially surjective, it
follows by general category theory that $S$ induces an equivalence of categories
for unipotent objects.
\end{proof} 

\section{Flat and Schottky bundles over complex
tori}\label{sec_Flat_and_Schottky_over_tori}

Let $X$ be a complex torus, and let $E$ be a vector bundle over $X$. 
The bundle $E$ is called homogeneous if $t_a^*E\cong E$ for all $a\in X$, where
$t_a:X\to X$, $t_a(x)=x+a$ denotes the translation-by-$a$ morphism.
The theorem of Matsushima and Morimoto can be stated as follows (see \cite{mat}
and \cite{mor}).
\begin{thmdefstyle}[Matsushima, Morimoto]
\label{thm:flat-homogeneous-unipotent} Let $E$ be a vector bundle over a complex
torus $X$. Then the following properties are equivalent:
\begin{itemize}
\item[a)] $E$ admits a flat connection.
\item[b)] $E$ admits a flat holomorphic connection.
\item[c)] $E$ is homogeneous.
\item[d)] Every indecomposable component of $E$ has the form $L\otimes E_{\rho}$
where $L$ is a line bundle of degree zero over $X$ and $\rho$ is a unipotent
representation.
\end{itemize}
\end{thmdefstyle}

As in the introduction, let us write $X=V/\Lambda$, and consider the
morphism $\alpha:\Lambda = \pi_{1}(X)\to \Sigma$ from which Schottky bundles
over $X$ are defined.

\begin{lem}
\label{lem:line-bundle-Schottky}Every line bundle of degree zero
over a complex torus is Schottky.\end{lem}
\begin{proof}
The proof is entirely analogous to that of Proposition
\ref{pro:unipotent-Schottky}.
Any degree zero line bundle on $X=V/\Lambda$ is necessarily flat,
so it is of the form $L_{\rho}$, for some $\rho:\pi_{1}(X)=\Lambda\to
GL_{1}(\mathbb{C})=\mathbb{C}^{*}$. As $\exp$ is surjective, we can find
$A_{j}\in \C$  for $j=1,...,g$, such that $\exp(A_{j})=\rho(e_{j})^{-1}$, and we
can define $\omega:=A_{1}dz_{1}+...+A_{g}dz_{g}$
and $f(z)=\exp(\int_{0}^{z}\omega)$. Then, the same computation as in
Proposition \ref{pro:unipotent-Schottky} gives $f(\lambda + z) \cdot
f(z)^{-1}=\rho(\lambda)^{-1}$
for all $\lambda\in I \cdot \Z^{g}$ which implies that the representation
defined by $\sigma(\lambda):=f(\lambda + z) \cdot f(z)^{-1} \cdot
\rho(\lambda)^{-1}$
factors over $\Sigma$, and produces a line bundle isomorphic to $L_{\rho}$. 
\end{proof}

As a consequence we can prove \textbf{Theorem \ref{thm:Flat-Schottky-vector}}:
\begin{proof}
Clearly, if $E$ is Schottky then it admits a flat connection. Conversely,
if $E$ admits a flat connection, then using Theorem
\ref{thm:flat-homogeneous-unipotent},
each of its indecomposable components is of the form $L\otimes E_{\rho}$
where $\rho$ is unipotent and $L$ a line bundle of degree zero.
Thus, by Theorem \ref{thm_main_for_tori} and Lemma
\ref{lem:line-bundle-Schottky},
$E_{\rho}$ and $L$ are both Schottky. The theorem then follows since
tensor products and direct sums of Schottky bundles are Schottky.
\end{proof}

\section{Unipotent and Schottky principal $G$-bundles}

\subsection{Principal and flat $G$-bundles}

Let $X$ be a complex manifold. Let $G$ be an \emph{affine
algebraic group over $\mathbb{C}$} (see for example \cite[Chapter I]{bor}) and
let $P$ be a holomorphic principal $G$-bundle over $X$.
  For definitions and elementary properties of holomorphic principal bundles,
 see for example \cite{Ba09}.

Recall that principal $G$-bundles over $X$ form a category
and in particular,
the notion of \emph{isomorphism} between principal $G$-bundles is
naturally defined. 

Recall also a few important constructions.
Given another complex affine algebraic group $H$, a principal $H$-bundle
over $X$ and a group homomorphism $\phi:H\to G$, we can construct
a principal $G$-bundle $\phi_{*}P$ over the same manifold $X$,
by defining it as $P\times G/\sim$ where $(p,g)\sim(p\cdot h,\phi(h)^{-1}g)$
for all $p\in P$, $h\in H$ and $g\in G$. The bundle $\phi_{*}P$
is said to be the \emph{$G$-bundle induced by the morphism $\phi$}.

One other related construction is that of an \emph{associated
vector bundle}. Given a representation of $H$, as a map $\psi:H\to GL(M)$,
where $M$ is some finite dimensional vector space, we can construct
the vector bundle $\psi_{*}P$ defined by $P\times M/\sim$ where
now $(p,m)\sim(p\cdot h,\psi(h)^{-1}m)$. The bundle $\psi_{*}P$
is called the \emph{vector bundle associated to the representation
$\psi$}. Since the two above constructions are intimately related,
it is clear that $\psi_{*}P$ can be also regarded as a principal
$GL(M)$-bundle.
\begin{Def}
Given a $G$-bundle $P$ and a subgroup $i:H\hookrightarrow G$, we
say that $P$ admits \emph{a reduction of structure group to $H$}
(or a $H$-structure) if there exists a $H$-bundle $P_{H}$ such
that $i_{*}P_{H}\cong P$ (as principal $G$-bundles). The bundle
$P_{H}$ is then called \emph{a reduction of $P$ to a $H$-bundle}.
\end{Def}
For example, a principal $G$-bundle is \emph{trivializable}, in the
sense that it is isomorphic to the trivial principal $G$-bundle given
by the projection onto the second factor $G\times X\to X$, if and
only if it admits a trivial structure, that is a reduction of structure
group to the trivial subgroup.
\begin{lem}
\label{lem:Induced-reduction}Let $H\subset G$ be a subgroup and
$G'$ another group. Let $P$ be a $G$-bundle and $\phi:G\to G'$
a homomorphism. If $P$ has a $H$-structure, then $\phi_{*}P$ has
a $\phi(H)$-structure.\end{lem}
\begin{proof}
Consider the commutative diagram\begin{eqnarray*}
H & \stackrel{\phi}{\to} & \phi(H)\\
i\downarrow &  & \downarrow j\\
G & \stackrel{\phi}{\to} & G'.\end{eqnarray*}
If $P$ has a $H$-structure then there exists $P_{H}$, a $H$-bundle
such that $i_{*}P_{H}\cong P$ as $G$-bundles. Then, $\phi_{*}P_{H}$
is a $\phi(H)$-bundle and $j_{*}$ of it is isomorphic to $\phi_{*}P$
as $G'$-bundles because \[
j_{*}(\phi_{*}P_{H})=(j\circ\phi)_{*}P_{H}=(\phi\circ
i)_{*}P_{H}=\phi_{*}i_{*}P_{H}\cong\phi_{*}P.\]
So, $\phi_{*}P$ has a $\phi(H)$-structure as wanted.
\end{proof}

Let $\pi_{1}(X)$ denote the fundamental group of $X$ (for some base point of
$X$). 

In the same way as with vector bundles, one can define natural bijections
between:
(a) principal $G$-bundles with a flat connection (flat $G$-bundles,
for short), (b) $G$-local systems, and (c) representations of $\pi_{1}(X)$
into $G$. By connection we will always mean holomorphic connection.

In particular, one has the following natural construction. For a given
representation $\rho\in\hom(\pi_{1}(X),G)$, we can associate a principal
$G$-bundle $P_{\rho}$ with a flat connection. It can be defined
using the principal $\pi_{1}(X)$-bundle $\tilde{X}\to X$, the universal
cover of $X$, as the $G$-bundle associated to the morphism $\rho:\pi_{1}(X)\to
G$
(the same construction as before, which works even though $\pi_{1}(X)$
is, of course, not an affine algebraic group). Conversely, given a
principal $G$-bundle $P$ with a flat connection, its holonomy representation
provides a natural representation (modulo conjugation by an element
in $G$). The construction of a vector bundle $E_{\rho}$ from a representation
$\rho:\pi_{1}(X)\to GL(M)$ is an instance of this construction.

If $P$ is a principal bundle with a flat connection, then for any
group morphism $\phi:G\to G'$ the associated bundle $\phi_{*}P$ receives
an induced flat connection. 


\subsection{Principal unipotent bundles}

Let us recall the definition of (algebraic) unipotent groups. Let $H$
be an affine algebraic group over $\mathbb{C}$. It has a faithful
linear representation into some general linear group, $\varphi:H\to GL(M)$,
for some finite dimensional $\C$-vector space $M$ (see \cite{bor}). An element
$h\in H$
is called unipotent if there is a natural number $n$ such that
$(\varphi(h)-1)^{n}=0$,
where $1\in GL(M)$ is the identity. Then $H$ is called unipotent
if all its elements are unipotent. It is known that any unipotent
algebraic group is isomorphic to a closed subgroup of the group of
upper triangular matrices with diagonal entries 1, and conversely
any such subgroup is unipotent (cf. Remark \ref{rem_Lekause_unipotent}). 

\begin{Def} A principal $G$-bundle
$P$ is called unipotent if there is a reduction of structure group
from $G$ to a unipotent subgroup $U\subset G$. 
\end{Def} 

Denote by $\Ad:G\to GL(\mathfrak{g})$ the adjoint representation of $G$
into its Lie algebra $\mathfrak{g}$. Given a principal $G$-bundle
$P$, the associated vector bundle $\Ad P=\Ad_{*}P$ will now play
an important role. Let $ZG$ denote the center of $G$, and note that
$G/ZG$ is again an affine algebraic group. It is clear that $\Ad$
factors through $G/ZG$, and denote by $\Ad_{0}:G/ZG\to GL(\mathfrak{g})$
the induced representation. 

Recall that, when $G$ is connected, the
homomorphism $\Ad_{0}:G/ZG\to GL(\mathfrak{g})$ is faithful, since
for a connected algebraic group $\ker\Ad=ZG$.

\begin{prop}
\label{pro:Principal-Vector}Let $P$ be a principal $G$-bundle and
assume that $ZG$ is trivial. Then, $P$ is unipotent if and only
if $\Ad P$ is a unipotent vector bundle.
\end{prop} 

\begin{proof}
Since $G$ has no center, $\Ad$ is faithful. Let $g_{ij}:U_{ij}\to G$
be (holomorphic) transition functions for $P$, so that $\Ad(g_{ij})$
are transition functions for $\Ad P$. Suppose that $\Ad P$ is a
unipotent vector bundle, i.e., a successive extension of the trivial
line bundle $\mathcal{O}_{X}$ over $X$. Then, wlog we can suppose
that there is a basis of $\mathfrak{g}$, the fiber of $\Ad P$, with
respect to which $\Ad(g_{ij})$ is upper triangular with 1's in the
diagonal. Then, $(\Ad(g)-1)^{n}=0$ which, by definition implies that
all the transition functions $g_{ij}$ have values in a unipotent
subgroup $U\subset G$. In turn, this means that $P$ has a $U$-structure
and is therefore unipotent. Conversely, if $P$ has a reduction of
structure group to a unipotent subgroup $U\subset G$ then $\Ad P$
has a reduction of structure group to $\Ad U$ because of Lemma
\ref{lem:Induced-reduction}.
And $\Ad U$ is clearly a unipotent subgroup of $GL(\mathfrak{g})$.
In this direction, we do not need to impose the condition $ZG=1$.\end{proof} 

We now introduce the following definition.

\begin{Def} \label{def:Ad-Unipotent}
Let $H$ be an affine algebraic group. An element $h\in H$
is called $\Ad$-unipotent if $(\Ad(h)-1)^{n}=0$ for some natural
number $n$. $H$ is called an $\Ad$-unipotent group if all its elements
are $\Ad$-unipotent. If $P$ is a principal $G$-bundle, then $P$
is called $\Ad$-unipotent if there is a reduction of structure group from
$G$ to an $\Ad$-unipotent subgroup $H\subset G$. \end{Def}

Now, we can show an analogue of the previous proposition, without
requiring that $G$ has a trivial center.

\begin{prop} \label{pro:Ad-Unipotent}Let $P$ be a principal $G$-bundle.
Then, $P$ is $\Ad$-unipotent if and only if $\Ad P$ is a unipotent
vector bundle.\end{prop} \begin{proof} If $P$ is $\Ad$-unipotent,
with reduction of structure group to an $\Ad$-unipotent subgroup
$H\subset G$, then $\Ad P$ has a reduction of structure group to
$\Ad H\subset GL(\mathfrak{g})$, which is easily verified to be a
unipotent subgroup. So $\Ad P$ is a unipotent vector bundle. Conversely,
let $g_{ij}:U_{ij}\to G$ be (holomorphic) transition functions for
$P$, so that $\Ad(g_{ij})$ are transition functions for $\Ad P$.
As before, if $\Ad P$ is a unipotent vector bundle, then, wlog we
can suppose that there is a basis of $\mathfrak{g}$, the fiber of
$\Ad P$, with respect to which $\Ad(g_{ij})$ is upper triangular
with 1's in the diagonal. Then, $(\Ad(g)-1)^{n}=0$ which, by definition
implies that all the transition functions $g_{ij}$ have values in
a $\Ad$-unipotent subgroup $H\subset G$, which means that $P$ has
a $H$-structure and is therefore $\Ad$-unipotent. \end{proof}

Recall that, given an element $g\in G$, the Jordan decomposition
theorem enables us to write $g=su$ where $s$ is semisimple and $u$
is unipotent.
\begin{prop}\label{prop:homogeneous_unipotent}
Let $X$ be a complex torus and let $E$ be a homogeneous vector bundle
over $X$ isomorphic to $\Ad P$ for a principal $G$-bundle $P$.
Then, $E$ is unipotent.\end{prop}
\begin{proof}
Let $E$ be homogeneous. By Theorem \ref{thm:flat-homogeneous-unipotent}
$E$ is flat, and since it is isomorphic to $\Ad P$ for some principal
$G$-bundle $P$, we can write $E=E_{\rho}$ for some representation
$\rho=\Ad\sigma:\pi_{1}(X)\to GL(\mathfrak{g})$ with $\sigma:\pi_{1}(X)\to G$.

Assume now that $E$ is indecomposable. Then by Theorem
\ref{thm:flat-homogeneous-unipotent}
$E\cong L\otimes U$ where $L$ is a degree zero line bundle and $U$
is unipotent. This means that $E=E_{\rho}=L_{\theta}\otimes U_{\tau}$
for some representations $\theta:\pi_{1}(X)\to Z(GL(\mathfrak{g}))$,
$\tau:\pi_{1}(X)\to U(GL(\mathfrak{g}))$ where $Z(GL(\mathfrak{g}))$,
$U(GL(\mathfrak{g}))$ denote, respectively the center of $GL(\mathfrak{g})$
and the subgroup of upper triangular matrices with 1's in the diagonal
(upon choosing a basis for $\mathfrak{g}$). It is clear that, for
any $\gamma\in\pi_{1}(X)$, $\rho(\gamma)=\theta(\gamma)\tau(\gamma)$
is the Jordan decomposition of $\rho(\gamma)$ inside $GL(\mathfrak{g})$.
Let $\sigma(\gamma)=su$, with $s$ semisimple and $u$ unipotent
in $G$. Since the Jordan decomposition is preserved under morphisms
and $\rho=\Ad\sigma$, we obtain $\Ad s=\theta(\gamma)$ and $\Ad u=\tau(\gamma)$.
An easy computation shows that if $\Ad s\in Z(GL(\mathfrak{g}))$
then $s\in ZG$. But then $1=\Ad s=\theta(\gamma)$. Since this is
true for all $\gamma\in\pi_{1}(X)$ we conclude that $L_{\theta}$
is the trivial line bundle, so $E$ is unipotent.

Finally, if $E$ is not necessarily indecomposable, the previous argument
shows $E$ is a direct sum of unipotent vector bundles, so it is unipotent.
\end{proof}

\subsection{Principal Schottky bundles}\label{subsec_Principal_Schottky_Bundles}

Let $P$ be a principal $G$-bundle over a complex manifold $X$ with fundamental
group $\pi_{1}(X)$. As before let $P_{\rho}$ be the flat $G$-bundle
associated with the representation $\rho\in\hom(\pi_{1}(X),G)$. Consider
a surjective moprhism $\alpha:\pi_{1}(X)\rightarrow\Sigma$
onto some free or free abelian group $\Sigma$. 
Recall that $ZG$ denotes the center of
a group $G$.
\begin{Def}\label{def:principal-schottky}
We say that $P$ is a $G$-principal \emph{Schottky} bundle with respect
to $\alpha$ (or $G$-Schottky bundle for short) if $P$ is isomorphic
to $P_{\rho}$ for some representation $\rho$ that satisfies:\[
\rho(\gamma)\in ZG,\quad\mbox{for all }\gamma\in\ker\alpha.\]
Such a representation is called a Schottky representation (w.r.t.
$\alpha$).
\end{Def}
For example, when $X$ is a Riemann surface of genus $g$, and 
the morphism $\alpha: \pi_{1}(X) \longrightarrow \Sigma$ is the one defined in
the
introduction, then a representation $\rho:\pi_{1}(X)\to G$ is Schottky if and
only if
$\rho(a_{j})\in ZG$ for all $j=1,...,g$. 
This is the motivating example for the previous definition.
\begin{rem}
By the correspondence between representations of $\pi_{1}(X)$
into $G$ and principal $G$-bundles with a flat connection, it is
clear that a necessary condition for a principal $G$-bundle to be
Schottky is that it admits a flat connection. \end{rem}
\begin{prop}
\label{pro:Ad-Schottky}Let $P$ be a principal $G$-bundle. If $P$
is $G$-Schottky (w.r.t. $\alpha$) then $\Ad P$ is a Schottky vector
bundle (w.r.t $\alpha$). Moreover, if we assume that $P$ admits
a flat connection, then the converse is also true.\end{prop}
\begin{proof}
Consider $P=P_{\rho}$ for some $\rho$ satisfying Definition
\ref{def:principal-schottky} for some
$\alpha:\pi_{1}(X)\twoheadrightarrow\Sigma$. By construction,
$\Ad P_{\rho}$ is the vector bundle $E_{\Ad\rho}$ associated to
the adjoint representation $\Ad\rho:\pi_{1}(X)\to GL(\mathfrak{g})$,
where $\mathfrak{g}$ is the Lie algebra of $G$. By the properties
of $\rho$, $\Ad\rho(\gamma)$ acts as the identity on $\mathfrak{g}$,
for all $\gamma \in Ker \alpha $. So $\Ad P_{\rho}$ is a Schottky vector
bundle with respect to $\alpha$. To prove the second statement, assume
that $P$ has a flat connection so that $P=P_{\rho}$ where $\rho$
is some representation $\rho:\pi_{1}(X)\to G$. Then $\Ad P\cong E_{\Ad\rho}$.
Suppose that $\Ad P$ is Schottky, so that $\Ad\rho(\gamma)=1$ for
all $\gamma\in\Sigma$. This implies that $\rho(\gamma)$ is in the
kernel of $\Ad$, which is $ZG$. So, $P$ is a $G$-Schottky bundle. \end{proof}
\begin{rem}
Let $P$ be a $\mathbb{C}^{*}$-bundle corresponding to a line bundle
of non-zero degree. Then, even though $\Ad P$ is trivial and has
the trivial flat connection, $P$ cannot admit a flat connection,
so the hypothesis that $P$ is flat is important in the second statement
of Proposition \ref{pro:Ad-Schottky}.
\end{rem}
The question whether all $G$-bundles which admit a flat connection
are Schottky is an important open problem, whose answer is not known
even in the case of Riemann surfaces. The situation is better when
$\pi_{1}(X)$ is abelian: when $X$ is an elliptic curve, all flat
vector bundles over $X$ are Schottky.

Let $X$ be a complex torus. 
We can now prove \textbf{Theorem \ref{thm:Flat-Schottky-principal}}:
\begin{proof} By definition Schottky bundles are flat. Conversely, let $P$
have a flat connection. Then $\Ad P$ has the induced flat connection,
as remarked earlier. Therefore, by Theorem \ref{thm:Flat-Schottky-vector},
$\Ad P$ is a Schottky vector bundle. So $P$ is a Schottky $G$-bundle,
by Proposition \ref{pro:Ad-Schottky}, this shows the claim.
\end{proof}
In particular, all flat principal bundles over elliptic curves are
Schottky.

\subsection{Other equivalent conditions for flat bundles}

From the work of Biswas-G\'omez \cite{bigo}, we can give a criterion for
flatness
using the adjoint vector bundle as follows. Given a degree $k$ polynomial
$q\in Sym^{k}(\mathfrak{g}^{*})$ invariant under the adjoint action
of $G$, we obtain a characteristic class of degree $k$ of the $G$-bundle
$P$, $c_{q}(P)\in H^{2k}(X,\mathbb{C})$, using the Chern-Weil construction.
\begin{prop}
Let $P$ be a principal $G$-bundle over a complex torus $X$ with
vanishing characteristic classes of degree one and two. Then, $P$
is flat (resp. homogeneous, resp. Schottky) if and only if $\Ad P$
is flat (resp. homogeneous, resp. Schottky).\end{prop}
\begin{proof}
If $P$ admits a flat connection (resp. is homogeneous or Schottky),
$\Ad P$ has the induced flat connection (resp. is homogeneous or
Schottky) (in this direction, we do not need the vanishing of characteristic
classes).
Conversely, let $\Ad P$ be flat (equivalently homogeneous or Schottky).
Then, by Theorem \ref{thm:flat-homogeneous-unipotent} $\Ad P$ is a direct
sum of vector bundles of the form $L\otimes U$, where $L$ is a line
bundle of degree zero and $U$ is unipotent. Therefore $\Ad P$ is
pseudostable (see \cite{bigo}). Therefore, by definition, $P$ is pseudostable.
Then, by \cite[Theorem 4.1]{bigo}, and using the vanishing of the characteristic
classes, $P$ is flat (and is homogeneous, and Schottky).
\end{proof}
Note that the one dimensional case (i.e, bundles over an elliptic
curve) was obtained by Azad-Biswas \cite{azbi}.
Note also, that Biswas-Subramanian \cite{bisu} showed that if $P$ be a principal
$G$-bundle with vanishing characteristic classes of degree one and
two (over any complex projective manifold) and $P$ is polystable,
then $P$ is flat (with a unitary connection!). In the Proposition
above, we are not assuming polystability.

\begin{cor}
Let $P$ be a principal $G$-bundle
over a complex torus.  If $P$ is flat then it is $\Ad$-unipotent.
Moreover, if the first and second Chern classes of $P$ vanish,
the converse also holds.
\end{cor}
\begin{proof}
 If $P$ is flat then $\Ad P$ is flat, since it has the induced connection. 
By Theorem \ref{thm:flat-homogeneous-unipotent} $\Ad P$ is a homogeneous.
By Proposition \ref{prop:homogeneous_unipotent} this implies that $\Ad P$ is
unipotent. By Proposition \ref{pro:Ad-Unipotent}
we obtain that $P$ is $\Ad$-unipotent.
Conversely, assume that $P$ is $\Ad$-unipotent. Then $\Ad P$ is unipotent,
by the same Proposition, which means $\Ad P$ is pseudostable. So,
by \cite[Theorem 4.1]{bigo}, we conclude that $P$ is flat.
\end{proof}

\begin{rem}
Note that if $P$ is unipotent, then it is flat, by \cite[paragraph
after Corollary 3.8]{bigo}.
\end{rem}

\textbf{Acknowledgement} We thank I. Biswas, J. P. Nunes and J. Mour\~ao for
interesting comments. This work was partially supported by CAMGSD at I.S.T., 
and by the Funda\c c\~ao para a Ci\^encia e a
Tecnologia through the programs
POCI/MAT/58549/2004 and PTDC/MAT/099275/2008.

\end{document}